\documentclass[10pt]{article}
\usepackage{amsmath, amssymb, amsthm, mathrsfs}
\usepackage{geometry}
\usepackage[colorlinks=true, allcolors=blue]{hyperref}
\title{A PDE perspective on the flat chain conjecture}
\author{Andrea Marchese \\ University of Trento}
\date{}

\theoremstyle{definition}
\newtheorem{definition}{Definition}[section]
\newtheorem{theorem}{Theorem}[section]
\newtheorem{conjecture}{Conjecture}[section]

\begin{document}

\maketitle

\begin{abstract}
This survey summarizes recent progress on the flat chain conjecture, which asserts the equivalence between metric currents and flat chains with finite mass in the Euclidean space. In particular, we focus on recent work showing that the conjecture is equivalent to a Lipschitz regularity estimate for a certain PDE.
\end{abstract}

\section{Introduction}

The study of \emph{currents}, that is, generalized surfaces arising as duals to smooth differential forms, has been central to geometric measure theory since the foundational work of Federer and Fleming, see \cite{Federer1969}. Two classes of currents are particularly significant: \textbf{normal currents}, which have finite mass and finite boundary mass, and \textbf{flat chains}, which are limits of normal currents in the flat norm. A long-standing problem, known as the \textbf{flat chain conjecture}, asks whether every \textbf{metric current} (in the sense of Ambrosio–Kirchheim \cite{AmbrosioKirchheim2000}) with compact support in \(\mathbb{R}^d\) corresponds to a Federer–Fleming flat chain.

The conjecture was first proved for $k=1$ by Schioppa \cite{Schioppa2016} and for $k=d$ by De Philippis and Rindler \cite{DePhilippisRindler2016}. It remains open for intermediate dimensions \(1 < k < d\). In this survey, we summarize recent progress on this problem.

%focusing on three key preprints:

%\begin{itemize} 
%   \item \textbf{Marchese–Merlo} \cite{MarcheseMerlo2024}: A new proof of the 1-dimensional case using Poincar\'e’s lemma, which bypasses the earlier technique of \emph{width functions}.
%  \item \textbf{Alberti–Bate–Marchese} \cite{AlbertiBateMarchese2025}: A study of the closability of differential operators, which provides a characterization of the tangent bundle of flat chains and their connection to metric currents.
%    \item \textbf{De Masi–Marchese} \cite{DeMasiMarchese2025}: A refined Lusin-type theorem for gradients and a conjecture for \(k\)-forms that would imply the full flat chain conjecture.
%\end{itemize}

%We also discuss the recent preprints by Takáč \cite{Takac2025}, Bate et al. \cite{BateEtAl2025}, and Arroyo-Rabasa--Bouchitt\'e \cite{ArroyoRabasaBouchitte2025}, which build on these ideas and settle related questions.

\section*{Acknoweldgments}
The author is partially supported by the PRIN project 2022PJ9EFL "Geometric Measure Theory: Structure of Singular Measures, Regularity Theory and Applications in the Calculus of Variations" CUP:E53D23005860006 and by GNAMPA-INdAM.

\section{Notation and preliminaries}

\subsection{Currents, mass, and flat norm}

We begin with the classical notions and notation for currents in the Euclidean space \(\mathbb{R}^d\), see \cite{Federer1969} for more details. Let $\mathscr{D}^k(\mathbb{R}^d)$ denote the space of smooth, compactly supported $k$-forms, endowed with the standard topology of test functions.

\begin{definition}[Current]
    A \(k\)-dimensional \textbf{current} \(T\) is a continuous linear functional on $\mathscr{D}^k(\mathbb{R}^d)$. The space of $k$-currents is denoted by $\mathscr{D}_k$.
\end{definition}

\begin{definition}[Mass]
    The \textbf{mass} of a current \(T\) is defined as:
    \[
    \mathbb{M}(T) = \sup \left\{ \langle T, \omega \rangle : \|\omega\|_{\infty} \leq 1 \right\},
    \]
    where $\|\omega\|_{\infty}$ is the comass norm.
\end{definition}

\begin{definition}[Boundary]
    The \textbf{boundary} of a \(k\)-current \(T\) is the \((k-1)\)-current \(\partial T\) defined by:
    \[
    \langle \partial T, \omega \rangle = \langle T, d\omega \rangle \quad \text{for all } \omega \in \mathscr{D}^{k-1}(\mathbb{R}^d).
    \]
\end{definition}

\begin{definition}[Normal current]
    A current \(T\) is \textbf{normal} if both \(\mathbb{M}(T) < \infty\) and \(\mathbb{M}(\partial T) < \infty\). The space of normal $k$-currents is denoted by $\mathscr{N}_k$.
\end{definition}

\begin{definition}[Flat norm, see \S 4.1.12 of \cite{Federer1969}]
    The \textbf{flat norm} of a current \(T\) can be defined in two equivalent ways:
    \begin{align*}
    \mathbb{F}(T) &= \inf \left\{ \mathbb{M}(R) + \mathbb{M}(S) : T = R + \partial S,\ R \in \mathscr{D}_k,\ S \in \mathscr{D}_{k+1} \right\} \\
    &= \sup \left\{ \langle T, \omega \rangle : \|\omega\|_\infty \leq 1, \|d\omega\|_\infty \leq 1 \right\}.
    \end{align*}
    A current is a \textbf{flat chain} if it is the limit of normal currents in the flat norm. The space of flat $k$-chains of finite mass is denoted by $\mathscr{F}_k$.
\end{definition}

In \cite{AlbertiMarchese2023}, Alberti and the author proved via an iteration of the standard \emph{polyhedral approximation theorem} that flat chains with finite mass are simply measurable pieces of normal currents.

\begin{theorem}[Structure of flat chains, see Theorem 1.1 of \cite{AlbertiMarchese2023}]\label{t:str_flat}

Let \(1 \leq k < d\), and let \(T \in \mathscr{F}_{k}(K)\) with \(\mathbb{M}(T)<\infty\). For every \(\varepsilon>0\) there exists a normal current \(T' \in \mathscr{N}_{k}(K)\) and a Borel set \(E\subset K\) such that
\begin{itemize}
    \item[(i)] \(\partial T'=0\),
    \item[(ii)] \(T = T' \llcorner E\), where $\llcorner$ denotes the restriction of the current to the set $E$,
    \item[(iii)] \(\mathbb{M}(T') \leq (2+\varepsilon)\,\mathbb{M}(T)\).
\end{itemize}
\end{theorem}
This result is the starting point of some recent work on the flat chain conjecture as it allows one to transfer properties of normal currents to flat chains.

\subsection{Metric currents}

Ambrosio and Kirchheim \cite{AmbrosioKirchheim2000} extended the theory of currents to complete metric spaces. Let $\text{Lip}(X)$ denote the space of real-valued Lipschitz functions on $X$, and $\text{Lip}_b(X)$ the subspace of bounded Lipschitz functions.

\begin{definition}[Metric current \cite{AmbrosioKirchheim2000}]
    A \(k\)-dimensional \textbf{metric current} \(T\) on a metric space \(X\) is a multilinear functional
    \[
    T: \text{Lip}_b(X) \times [\text{Lip}(X)]^k \to \mathbb{R}
    \]
    satisfying:
    \begin{enumerate}
        \item \textbf{Continuity}: \(T(f, \pi_1, \dots, \pi_k)\) is continuous under pointwise convergence of \(\pi_i\) with uniformly bounded Lipschitz constants.
        \item \textbf{Locality}: \(T(f, \pi_1, \dots, \pi_k) = 0\) if some \(\pi_i\) is constant on a neighborhood of \(\operatorname{supp}(f)\).
        \item \textbf{Finite mass}: There exists a finite measure $\mu$ such that
        \[
        |T(f, \pi_1, \dots, \pi_k)| \leq \prod_{i=1}^k \text{Lip}(\pi_i) \cdot \int |f| \, d\mu.
        \]
    \end{enumerate}
\end{definition}

In \(\mathbb{R}^d\), every metric current $T$ induces a classical current $\widetilde{T}$ via the formula \cite[Theorem 11.1]{AmbrosioKirchheim2000}:
\[
\langle \widetilde{T}, \omega \rangle = \sum_{\mathbf{i} \in I(d,k)} T(\omega_{\mathbf{i}}, x_{i_1}, \ldots, x_{i_k}),
\]
where $\mathbf{i}=(i_1,\dots,i_k)$ is a multi-index of length $k$ in $\mathbb{R}^d$ and $\omega = \sum_{\mathbf{i}} \omega_{\mathbf{i}} dx^{i_1} \wedge \dots \wedge dx^{i_k}$.
The \textbf{flat chain conjecture} asserts that if \(T\) is a metric current with compact support, then $\widetilde{T} \in \mathscr{F}_k$.

\subsection{The decomposability bundle and its $k$-dimensional analogue}

A key tool in the study of differentiability of Lipschitz functions with respect to measures is the \textbf{decomposability bundle} $V(\mu, x)$, introduced by Alberti and the author in \cite{AlbertiMarchese2016}.

\begin{definition}[Decomposability bundle, see Section 6.1 of \cite{AlbertiMarchese2016}]
    For a Radon measure \(\mu\) on $\mathbb{R}^d$, the \textbf{decomposability bundle} \(V(\mu, \cdot)\) is a Borel map whose values are vector spaces, defined in such a way that for $\mu$-a.e. point $x$ a vector \(v \in V(\mu, x)\) if and only if there exists a 1-normal current \(N\) with \(\partial N = 0\) such that
    \[
    \lim_{r \to 0} \frac{\mathbb{M}((N - v\mu) \llcorner B(x,r))}{\mu(B(x,r))} = 0,
    \]
    where $\llcorner$ denotes the restriction of the current to the ball.
\end{definition}
This bundle characterizes the directions in which Lipschitz functions are $\mu$-a.e. differentiable. Its generalization to $k$-vectors is natural. Let $\Lambda_k(\mathbb{R}^d)$ denote the space of $k$-vectors on $\mathbb{R}^d$.

\begin{definition}[$k$-Tangent bundle, see Definition 4.1 of \cite{AlbertiMarchese2023}]
    For a Radon measure \(\mu\) on $\mathbb{R}^d$, the \textbf{$k$-tangent bundle} \(V_k(\mu, \cdot)\) is Borel map whose values are vector subspaces of $\Lambda_k(\mathbb{R}^d)$, defined in such a way that for $\mu$-a.e. point $x$ a $k$-vector \(v \in V_k(\mu, x)\) if and only if there exists a $k$-normal current \(N\) with \(\partial N = 0\) such that
    \[
    \lim_{r \to 0} \frac{\mathbb{M}((N - v\mu) \llcorner B(x,r))}{\mu(B(x,r))} = 0.
    \]
\end{definition}

It follows from Theorem \ref{t:str_flat} that this bundle characterizes flat chains, as shown in \cite{AlbertiMarchese2023}:
\begin{theorem}[Characterization of flat chains, see Theorem 1.2 of \cite{AlbertiMarchese2023}]\label{t:char_flat}
Let $\tau \in L^1(\mu; \Lambda_k(\mathbb{R}^d))$ be a $k$-vector field. Then the current $T = \tau \mu$ is a flat chain ($T \in \mathscr{F}_k$) if and only if $\tau(x) \in V_k(\mu, x)$ for $\mu$-a.e. $x$.
\end{theorem}

The connection to metric currents is established in the following result, which links the continuity property of a metric current to the geometry of the decomposability bundle of its mass measure. We refer to Section 5.8 of \cite{AlbertiMarchese2016} for the definition of span of a $k$-vector.

\begin{theorem}[Metric currents and the decomposability bundle, see Theorem 5.8 of \cite{AlbertiBateMarchese2025}]
Let $T$ be a metric $k$-current on $\mathbb{R}^d$ with compact support and let $\mu$ be its mass measure. Then there exists a bounded $k$-vector field $\tau$ such that $\operatorname{span}(\tau(x)) \subset V(\mu, x)$ for $\mu$-a.e. $x$ and
\[
T(f, \pi_1, \dots, \pi_k) = \int f \langle \tau, d_{V}\pi_1 \wedge \dots \wedge d_{V}\pi_k \rangle d\mu,
\]
where $d_V$ denotes the differential relative to the subspace $V(\mu,x)$\footnote{The fact that Lipschitz functions admit such differential $\mu$-a.e. is part of the main result of \cite{AlbertiMarchese2016}.}. Conversely, if $\tau(x) \in V_k(\mu,x)$ for $\mu$-a.e. $x$, then the above formula defines a metric current.
\end{theorem}

\section{A new proof of the 1-dimensional case}

The first proof of the 1-dimensional conjecture was given by Schioppa \cite{Schioppa2016}, using Alberti representations and width functions. In \cite{MarcheseMerlo2024}, Merlo and the author provide a new proof which bypasses such notions and instead uses only \textbf{Poincaré’s lemma} and elementary functional-analytic arguments.

We first need two definitions central to this new approach.

\begin{definition}[Orthogonal bundle, see Definition 3.1 of \cite{MarcheseMerlo2024}]
    For a subspace \(V \subset \Lambda_k(\mathbb{R}^d)\), define its \textbf{mass-orthogonal complement} as
    \[
    V^\perp = \left\{ \tau \in \Lambda_k(\mathbb{R}^d) : \|\tau\| \leq \|\tau + \sigma\| \text{ for all } \sigma \in V \right\},
    \]
    where $\|\cdot\|$ denotes the mass norm on $k$-vectors.
\end{definition}

\begin{definition}[Purely non-flat current, see Definition 3.1 of  \cite{MarcheseMerlo2024}]
    A current $T = \tau \mu$ is \textbf{purely non-flat} if $\tau(x) \in V_k(\mu, x)^\perp$ for $\mu$-a.e. $x$.
\end{definition}

For such currents, the flat norm is determined by the action on closed forms and coincides with the mass, a result based on an application of the Hahn-Banach theorem.

\begin{theorem}[Flat norm of purely non-flat currents, see Proposition 3.2 and Proposition 3.3 of \cite{MarcheseMerlo2024}]\label{t:pnf}
    If \(T\) is a purely non-flat \(k\)-current with compact support, then
    \[
    \mathbb{F}(T) = \mathbb{M}(T) = \mathbb{F}_0(T),
    \]
    where \(\mathbb{F}_0(T) = \sup \left\{ \langle T, \omega \rangle : \|\omega\|_{\infty} \leq 1,\ d\omega = 0 \right\}\) is the \textbf{closed flat seminorm}.
\end{theorem}
\begin{proof}[Proof Sketch]
The inequality $\mathbb{F}(T) \leq \mathbb{M}(T)$ is always true. Toward a proof of the inequality $\mathbb{M}(T) \leq \mathbb{F}(T)$, one assumes by contradiction that there exists a decomposition $T = R + \partial S$ with $\mathbb{M}(R) + \mathbb{M}(S) < \mathbb{M}(T)$. Using the fact that $T$ is purely non-flat and Theorem \ref{t:char_flat}, one can show that the part of $R$ orthogonal to $V_k(\mu,x)$ must have mass at least $\mathbb{M}(T)$, leading to a contradiction. The equality $\mathbb{F}(T) = \mathbb{F}_0(T)$ follows from a Hahn-Banach argument.
\end{proof}

\subsection{Sketch of the proof of the 1D conjecture}

\begin{enumerate}
    \item Assume by contradiction that there exists a metric 1-current \(T\) such that \(\widetilde{T}\) is not a flat chain.
    \item Decompose \(\widetilde{T}\) into a flat part and a purely non-flat part $T_n$ using the projection onto $V_1(\mu, x)$.
    \item By Theorem \ref{t:pnf}, $\mathbb{F}(T_n) = \mathbb{M}(T_n) > 0$.
     \item Translate $T_n$ by a vector $v$. Since the mass measure of $T_n$ is singular wrt. Lebesgue, then for a.e. $v$, the translated current $\tau_{v\sharp}T_n$ and $T_n$ are mutually singular, so that $\mathbb{M}(T_n - \tau_{v\sharp}T_n) = 2\mathbb{M}(T_n)$.
    \item For any closed 1-form \(\omega\) with \(\|\omega\| \leq 1\), by Poincaré’s lemma, there exists a Lipschitz function $\pi$ such that $d\pi = \omega$ and \(\text{Lip}(\pi) \leq 1\). Apply this to the forms $\omega_v$ realizing the closed flat seminorm of $T_n - \tau_{v\sharp}T_n$ obtaining corresponding functions $\pi_v$ which, up to subsequences, converge to a function $\pi$ by Arzelà-Ascoli.
    \item Again by Theorem \ref{t:pnf}, $T_n(1,\pi_v) - T_n(1,\pi_v\circ\tau_v) = \mathbb{F}_0(T_n - \tau_{v\sharp}T_n)=\mathbb{M}(T_n - \tau_{v\sharp}T_n) = 2\mathbb{M}(T_n)$.
    \item However, by the continuity axiom of metric currents, the action of $T_n$ on $(1,\pi_v)$ and $(1,\pi_v \circ \tau_v)$ must be close if $v$ is small, leading to a contradiction with the previous point.
\end{enumerate}

This proof reveals that the higher-dimensional case would follow if one could replace Poincaré's lemma by a suitable $L^\infty$-to-Lipschitz estimate for the PDE $d\phi = \omega$ for $k$-forms: a known challenge due to the failure of Schauder estimates for continuous data.

\section{A Lusin-type theorem and the full conjecture}

The classical Lusin theorem states that a measurable function coincides with a continuous one outside of a set of arbitrarily small measure. Alberti \cite{Alberti1991} proved a deep analogue for gradients: any bounded vector field $f$ coincides with the gradient of a $C^1$ function outside an arbitrarily small set. This was later extended to general measures by the author and Schioppa \cite{MarcheseSchioppa2019}.

In \cite{DeMasiMarchese2025}, De Masi and the author proved a refined version where the estimate does not degenerate on the small set if the vector field is orthogonal to the decomposability bundle.

\begin{theorem}[Refined Lusin theorem for gradients, see Theorem 1.1 of \cite{DeMasiMarchese2025}]
    Let \(\mu\) be a Radon measure, \(\Omega \subset \mathbb{R}^n\) open with \(\mu(\Omega) < \infty\), and \(f: \Omega \to \mathbb{R}^n\) Borel with \(f(x) \in V(\mu, x)^\perp\) \(\mu\)-a.e. Then for every \(\varepsilon > 0\), there exists a \(C^1\) function \(g\) and a compact set \(K \subset \Omega\) such that:
    \begin{itemize}
        \item \(\mu(\Omega \setminus K) < \varepsilon\),
        \item \(Dg = f\) on \(K\),
        \item \(\|Dg\|_{L^p(\mu)} \leq (1+\varepsilon)\|f\|_{L^p(\mu)}\) for all $p \in [1, \infty]$.
    \end{itemize}
\end{theorem}

This result is sharp. The natural next step is to generalize this from gradients (i.e., $1$-forms) to general $k$-forms.

\begin{conjecture}[Lusin for \(k\)-forms, see Conjecture 4.1 \cite{DeMasiMarchese2025}]\label{c:demama}
    Let $\omega$ be a $k$-form such that $\langle \omega(x), \tau \rangle = 0$ for all \(\tau \in V_k(\mu, x)\) and \(\mu\)-a.e. \(x\). Then for every \(\varepsilon > 0\), there exists a \(C^1\) \((k-1)\)-form \(\phi\) and a compact set $K$ such that:
    \begin{itemize}
        \item \(\mu(\Omega \setminus K) < \varepsilon\),
        \item \(d\phi = \omega\) on \(K\),
        \item \(\operatorname{Lip}(\phi_\mathbf{i}) \leq C(n) \|\omega\|_{L^\infty(\mu)}\) for all components $\phi_\mathbf{i}$ of $\phi$.
    \end{itemize}
\end{conjecture}

De Masi and the author showed in Section 4 of \cite{DeMasiMarchese2025} that this conjecture would imply the full flat chain conjecture. The idea is that if a metric current were not flat, its purely non-flat part would have a tangent field in $V_k(\mu, x)^\perp$. The validity of Conjecture \ref{c:demama} would then allow the construction of test forms that violate the continuity property of metric currents, similar to the argument in the 1-dimensional case.

\section{Takáč’s counterexample to Lang's conjecture}

The version of the flat chain conjecture originally considered by Lang \cite{Lang2011} did not include the finite mass assumption. In \cite{Takac2025}, Jakub Takáč constructs a counterexample to this stronger version. His construction is based on the failure of estimates for the prescribed Jacobian equation, $\det D\phi = f$, which is deeply connected to the failure of $L^\infty$-to-Lipschitz estimates for the equation $d\omega = \phi$ in higher dimensions, see \cite{Ornstein1962, McMullen1998}.

The analogy between the approaches by \cite{MarcheseMerlo2024} and \cite{DeMasiMarchese2025} and the approach by \cite{Takac2025} lies in the fact that both problems require constructing a function with controlled Lipschitz constant whose derivative matches a given field or form on a large set. Takáč's result shows that this is impossible in general. This is in contrast with the Lusin-type conjecture of \cite{DeMasiMarchese2025} and underlines the fact that the finite mass assumption provides additional \emph{flexibility}, making the critical $L^\infty$-to-Lipschitz estimate potentially hold in a measure theoretic sense rather than pointwise. 

Although Tak\'a\v{c}'s preprint \cite{Takac2025} appeared after \cite{MarcheseMerlo2024}, the fundamental connection between the flat chain conjecture and PDE regularity issues was developed independently through both lines of investigation and with different techniques.

\section{Recent related work and final comments}

The structure of 1-dimensional metric currents has been further analyzed in three very recent preprints, which improve our understanding of the 1-dimensional case and open new avenues for exploring the structure of metric currents.

In \cite{AmbReVi} it is proven that every locally normal metric 1-current (in the sense of Lang and Wenger) can be written as a superposition of curves with locally finite length, generalizing a result by Smirnov in the Euclidean setting, see \cite{Smi} and by Paolini and Stepanv in the metric setting, see \cite{PaoSte}.

In \cite{BateEtAl2025} and \cite{ArroyoRabasaBouchitte2025} it is proven that every metric 1-current with finite mass can be approximated in mass by a sequence of \emph{normal} 1-currents. This is obtained through a very strong analogue of Theorem \ref{t:str_flat} in the metric setting.\\

The flat chain conjecture has seen significant progress in recent months. The shift from geometric constructions (width functions) to PDE and closability arguments has proven powerful, leading to new proofs and even stronger results in dimension 1.

The full conjecture for \(k > 1\) remains open, but the works described in this survey suggest promising approaches, potentially leading either to a complete proof or to further counterexamples: a Lusin-type theorem for \(k\)-forms.

\bibliographystyle{alpha}
\bibliography{references}

\end{document}